\documentclass{amsart}

\usepackage{amsmath,amssymb,amsthm}
\usepackage{hyperref}
\usepackage{xcolor}

\newtheorem{prop}{Proposition}[section]
\newtheorem{thm}[prop]{Theorem}
\newtheorem{lem}[prop]{Lemma}
\newtheorem{cor}[prop]{Corollary}

\theoremstyle{definition}

\newtheorem{rem}[prop]{Remark}

\newtheorem*{ack}{Acknowledgement}

\def\co{\colon\thinspace}

\newcommand{\C}{\mathbb C}

\newcommand{\rmd}{\mathrm d}

\newcommand{\rme}{\mathrm e}

\newcommand{\rmi}{\mathrm i}

\newcommand{\N}{\mathbb N}

\newcommand{\R}{\mathbb R}

\newcommand{\lra}{\longrightarrow}
\newcommand{\ra}{\rightarrow}

\DeclareMathOperator{\Int}{\mathrm{Int}}

\DeclareMathOperator{\st}{\mathrm{st}}

\begin{document}

\author{Marc Kegel}
\address{Institut f\"ur Mathematik, Humboldt-Universit\"at zu Berlin,
Unter den Linden 6, D-10099 Berlin, Germany}
\email{kegemarc@math.hu-berlin.de}
\author{Jay Schneider}
\address{Zweitag GmbH, Alter Fischmarkt 12, D-48143 M\"unster, Germany}
\email{jay.schneider@zweitag.de}
\author{Kai Zehmisch}
\address{Mathematisches Institut, Justus-Liebig-Universit\"at Gie{\ss}en,
Arndtstra{\ss}e 2, D-35392 Gie{\ss}en, Germany}
\email{Kai.Zehmisch@math.uni-giessen.de}

\title[Symplectic dynamics and the 3-sphere]{Symplectic dynamics and the 3-sphere}

\date{13.03.2019}

\begin{abstract}
  Given a knot in a closed connected
  orientable $3$-manifold
  we prove that if the exterior of the knot
  admits an aperiodic contact form
  that is Euclidean near the boundary,
  then the $3$-manifold is diffeomorphic to the $3$-sphere
  and the knot is the unknot.
\end{abstract}

\subjclass[2010]{53D35; 37C27, 37J55, 57M25.}
\thanks{{\it Key words and phrases.}
Weinstein conjecture,
symplectic dynamics,
$3$-sphere,
aperiodic Reeb vector field,
contact embedding,
transverse knots,
diffeomorphism type}
\thanks{{\it Add in proof.}
We learned from the referee
that Vincent Colin has announced
the result in a talk on June 18, 2018.}

\maketitle


\section{Introduction\label{sec:intro}}

We consider compact contact manifolds with boundary and ask:
What does the boundary know about the interior?
The question can be answered meaningfully
in terms of the diffeomorphism type,
if one imposes boundary conditions on the contact structure
and if one only allows contact forms whose Reeb vector field
has a particular dynamical behaviour.

A contact form is called {\bf aperiodic}
if the associated Reeb vector field
does not have any periodic solution.
The existence of an aperiodic contact form
on a compact contact manifold with boundary
that is of a certain standard form near the boundary
implies strong restrictions
on the underlying topology.
If for instance the boundary of the contact manifold $M$
is $S^2=\partial M$
Eliashberg--Hofer \cite{eh94} proved a
{\bf global Darboux theorem}
via holomorphic disc filling:
If a neighbourhood of $\partial M$
contactomorphically is equal to a neighbourhood of
$S^2=\partial D^3$ in $D^3$
and if the standard contact form
\[
\alpha_{\st}:=\rmd z+\frac12\big(x\rmd y-y\rmd x\big)
\]
extends to an aperiodic contact form $\alpha$ on $M$
then $(M,\alpha)$ admits a strict contact embedding into 
$(\R^3,\alpha_{\st})$.
In particular,
by Alexander's theorem $M$ must be a ball
and the Reeb flow of $\alpha$
does not admit any trapped orbits.

The relation from symplectic dynamics
to the diffeomorphism type
is obtained via holomorphic disc filling.
This way odd-dimensional balls \cite{eh94,gz16b}
and neighbourhoods $T^d\times D^{2n+1-d}$, $d<n$,
of subcritical isotropic tori \cite{bschz}
have been characterised
in terms of aperiodic contact forms standard near the boundary.
In this article we reduce the case of
neighbourhoods $S^1\times D^2$ of transverse knots in a contact $3$-manifold
to the global Darboux theorem
due to Eliashberg--Hofer \cite{eh94}.
Furthermore we give a characterisation of the $3$-sphere in terms
of the symplectic dynamics of a class of Reeb vector
fields on transverse knot exteriors, see Theorem \ref{thm:diffeotos3}.


\section{Aperiodic exterior of a transverse knot\label{sec:aperextranskn}}

Let $(M,\xi)$ be a compact connected co-oriented contact $3$-manifold
whose boundary $\partial M$ is diffeomorphic to the $2$-torus.
Assume that $\partial M$ admits a contact embedding into $\R^3$
which is equipped with the standard contact structure
$\xi_{\st}$ given by the kernel of $\alpha_{\st}$.
By a contact embedding of $\partial M$
into $(\R^3,\xi_{\st})$
we mean the following:

\subsection{Definition}
\label{defofboundcon}

$\partial M$ admits a {\bf contact embedding} into $(\R^3,\xi_{\st})$
if there exists a co-orientation preserving
contact embedding $\varphi$
of a neighbourhood $U\subset(M,\xi)$ of $\partial M$ 
into $(\R^3,\xi_{\st})$ such that:
\begin{enumerate}
\item[(a)]
 The interior of $\varphi(U)$ is mapped into the
 bounded component of $\R^3\setminus\varphi(\partial M)$.
\item[(b)]
 Affine lines parallel to the $z$-axis
 intersect $\varphi(\partial M)$ in at most two points.
 In the case of exactly two points
 we require the intersections to be transverse;
 in the case of a single point the intersection has to be tangent
 and is called {\bf point of tangency}.
\item[(c)]
The points of tangency form a disjoint union of two embedded circles
that divide $\varphi(\partial M)$ into an {\bf upper} and a {\bf lower part}.
Additionally,
the tangency circles are
nowhere tangent to a line parallel to the $z$-axis.
\end{enumerate}

\begin{rem}
The existence of $\varphi$
and condition (a) allow a gluing construction
that we call Euclidean completion.
This is discussed below.
The condition (b) is of dynamical nature
and allows to prove Proposition \ref{prop:aperieukl} (i)
via the global Darboux theorem.
Similar conditions were used in \cite{bschz,eh94,gz16b}.
Conditions (c) delivers odd-symplectic models for $(M,\rmd\alpha)$
formulated in Proposition \ref{prop:aperieukl} (ii).
\end{rem}

Let $\alpha$ be a contact form on $M$
whose kernel defines $\xi$.
We call $\alpha$ {\bf standard near the boundary}
if the restriction of $\varphi$ to a possibly smaller neighbourhood
$U$ of $\partial M$ pulls $\alpha_{\st}$ back
to $\alpha|_U$.
We remark that
the space of contact forms on $(M,\xi)$ standard
near the boundary is not empty.
This is because the space of positive contact forms on $M$
that define the same co-oriented contact structure $\xi$
can be identified with the space of positive smooth functions on $M$.
This allows convex interpolation
of $\xi$-defining contact forms.

The {\bf Euclidean completion} $\widehat{\R^3}$ of $(M,\xi)$
is the result of gluing $M$ to the closure of
the unbounded component of
$\R^3\setminus\varphi(\partial M)$
via $\varphi$.
It admits a contact form $\hat{\alpha}$ given by
$\alpha$ on $M$ and $\alpha_{\st}$
on the unbounded component of $\R^3\setminus\varphi(\partial M)$.
The induced contact structure $\hat{\xi}$ on $\widehat{\R^3}$
is independent of the choice of $\alpha$.

\subsection{A global model}

Denote by $\R^2\times\{z_*\}$ the affine plane of $\R^3$
that touches $\varphi(\partial M)$ from below.
Let $A\equiv A\times\{z_*\}$ be the annulus in $\R^2\times\{z_*\}$
that is obtained by projecting $\varphi(\partial M)$
to $\R^2\times\{z_*\}$ along the $z$-axis.
Observe that $A$ is bounded by two embedded circles
inside $\R^2\times\{z_*\}$.

Denote by $V$ the solid torus
that is given by the closure of the bounded component of
$\R^3\setminus\varphi(\partial M)$.
We equip $V$ with the odd-symplectic form
$\rmd x\wedge\rmd y$.
For the notion of odd-symplectic forms
-- symplectic forms on odd-dimensional manifolds --
the reader is referred to \cite{gz18}.

\begin{prop}
 \label{prop:aperieukl}
 Let $(M,\xi)$ be a compact connected contact manifold
 with boundary $T^2$ that admits a contact embedding
 into $(\R^3,\xi_{\st})$ as described in Definition \ref{defofboundcon}.
 Let $\alpha$ be a contact form on $(M,\xi)$
 that is standard near the boundary.
 If $\alpha$ is aperiodic, then
 \begin{enumerate}
 \item[(i)]
 there exists a strict contact embedding of
 $(M,\alpha)$ into $(A\times\R,\alpha_{\st})$
 so that in particular $\xi$ is tight
 and $M$ is diffeomorphic to $S^1\times D^2$.
 \item[(ii)]
 $(M,\rmd\alpha)$ is odd-symplectomorphic to
 the solid torus $(V,\rmd x\wedge\rmd y)$ in $\R^3$.
 \end{enumerate}
\end{prop}

The proposition applies for example if $\varphi(\partial M)$
is a rotationally symmetric torus
that is obtained by rotating a round circle
in the $xz$-plane around the $z$-axis.
Furthermore by Alexander's theorem, each embedded
$2$-torus bounds a solid torus in $\R^3$ so that 
all $M$ are diffeomorphic to $S^1\times D^2$ that satisfy
conditions (a)-(b) in Definition \ref{defofboundcon}.

\begin{proof}[{\bf Proof of Proposition \ref{prop:aperieukl}}]
 We identify the annulus $A$ with its image in $\widehat{\R^3}$
 under the inclusion of the closure
 of the unbounded component of $\R^3\setminus\varphi(\partial M)$.
 Denoting the Reeb vector field
 of the contact form $\hat{\alpha}$ by $\hat{R}$
 we observe that $\hat{R}=\partial_z$
 in a neighbourhood of $\partial A\times\R$.
 In particular,
 the flow $\varphi_t$ of $\hat{R}$
 restricts to a global flow
 on the Euclidean completion $\widehat{A\times\R}$,
 which is defined similarly to $\widehat{\R^3}$.
 
 The contact form $\hat{\alpha}$
 on the Euclidean completion $\widehat{\R^3}$
 is aperiodic by condition (b) in Definition \ref{defofboundcon}.
 Therefore,
 the Reeb vector field $\hat{R}$ of $\hat{\alpha}$
 does not have any trapped orbits in
 forward and backward time.
 This is a consequence of Eliashberg--Hofer's
 global Darboux theorem \cite[Theorem 2]{eh94}.
 In fact,
 invoking the computations on 
 \cite[p.~1306]{eh94}
 the map $\psi(x,y,z)=\varphi_z(x,y,z_*)$
 defines a diffeomorphism from $A\times\R$
 onto $\widehat{A\times\R}$
 such that $\psi^*\hat{\alpha}=\alpha_{\st}$.
 Composing the strict contact embeddings
 $(M,\alpha)\subset(\widehat{A\times\R},\hat{\alpha})$ and
 $\psi^{-1}\co(\widehat{A\times\R},\hat{\alpha})\ra(A\times\R,\alpha_{\st})$
 shows (i).
 
 In order to show (ii) observe that
 the flow of any vector field $X=f\partial_z$ parallel to $\partial_z$
 is odd-symplectic,
 i.e. sends $\rmd\alpha_{\st}$ to itself,
 as $L_X\alpha_{\st}=\rmd f$.
 Using conditions (b)-(c) in Definition \ref{defofboundcon},
 the required odd-symplectomorphism is obtained
 by interpolating $\psi^{-1}(\partial M)$
 and $\varphi(\partial M)$ along the flow lines of $\partial_z$.
 For that observe
 that $\varphi(\partial M)$ and $A\times\R$
 as well as $\psi^{-1}(\partial M)$ and $A\times\R$
 intersect along two embedded circles
 that divide both tori into upper and lower parts.
 The lower parts admit neighbourhoods
 along which both tori coincide.
 The upper part of $\psi^{-1}(\partial M)$ can be brought 
 to the upper part of $\varphi(\partial M)$ by following $\partial_z$
 with speed equal to the oriented length of the line segments
 parallel to the $z$-axis in between.
 This can be done without moving the closure of the lower parts of the tori.
 The map
 $F\co M\subset\widehat{A\times\R}\ra V\subset A\times\R$
 obtained from $\psi^{-1}|_M$
 followed by the time-1 map of the diffeotopy
 induced by the interpolation is odd-symplectic as required.
\end{proof}

\begin{rem}
 The proof of Proposition \ref{prop:aperieukl} (ii) shows
 that $\varphi(\partial M)$ can be interpolated to any
 embedded $2$-torus in $\R^3$ that satisfies conditions
 (b)-(c) in Definition \ref{defofboundcon}
 so that the respective circles of tangency coincide
 and both tori are equal on neighbourhoods of them.
 Therefore,
 the enclosed solid tori equipped with $\rmd x\wedge\rmd y$
 and $(M,\rmd\alpha)$ are odd-symplectomorphic.
\end{rem}

\begin{rem}
 The Eliashberg--Hofer disc filling argument in \cite{eh94}
 only requires the non-existence of contractible periodic
 Reeb orbits for $\hat{\alpha}$.
 Therefore,
 it suffices to assume
 that $\alpha$ does not have any periodic Reeb orbit in $M$
 that is contractible or
 homotopic to one of the circles of tangency in $\partial M$.
\end{rem}

\begin{rem}
 \label{rem:etnyre-ghrist}
 If $(M,\xi)$
 is overtwisted and allows a Reeb vector field $R$
 that is tangent to $\partial M$,
 then $R$ admits a contractible periodic orbit
 contained in $M\setminus\partial M$,
 see \cite{eg02,hof93}.
 Part (i) of Proposition \ref{prop:aperieukl} yields the result
 with Euclidean boundary condition.
\end{rem}

\subsection{Tight contact structures on solid tori\label{sec:tightcontsoltor}}

 By construction $F$ fixes the closure of
 the lower part of
 $\partial M\equiv\varphi(\partial M)\equiv\partial V$ pointwise.
 Therefore,
 the restriction of $F$ to the boundary
 is a Dehn twist on the upper part
 up to isotopy
 and may not map the meridian to the {\bf meridian},
 which by definition is the non-trivial homology class
 that lies in the image of the connecting homomorphism
 for the pair $(M,\partial M)$ and $(V,\partial V)$, resp.
 The homology class in $\partial M$
 represented by a boundary circle of the lower part
 is chosen to be the {\bf longitude}.
 
 The method used to prove Proposition \ref{prop:aperieukl} (ii)
 does not yield a contactomorphism if $f$ is non-constant
 as the above Lie derivative $L_X\alpha_{\st}=\rmd f$ is closed
 and $\xi_{\st}$ non-integrable.
 On the other hand,
 Makar-Limanov proved a
 classification \cite[Theorem 5.4]{mak98}
 for tight contact structures on solid tori in the case
 the characteristic foliation on the boundary $2$-torus
 does not have any singular points and Reeb components.
 Applying this,
 $(M,\xi)$ is contactomorphic to $(V,\xi_{\st})$
 if the restriction of $F$ to the boundary
 is isotopic to the identity.


\section{A characterisation of the $3$-sphere\label{sec:achof3sphere}}

With Section \ref{sec:aperextranskn}
we can characterise the unknot and the $3$-sphere
in terms of symplectic dynamics of knot exteriors.
Before we start with this
we will give the analog construction
for the exterior of a point using symplectic dynamics.

Given a connected oriented compact $3$-manifold $M$
with boundary $S^2$ Martinet's theorem ensures existence
of a contact structure $\xi$ on $M$
such that a neighbourhood of $\partial M$ in $(M,\xi)$
can be identified with a neighbourhood of $S^2=\partial D^3$
in $(D^3,\xi_{\st})$
contactomorphically.
Indeed, the closed manifold $M'$
obtained by gluing $D^3$ to $M$ along $S^2$
carries a contact structure $\xi'$.
Darboux's theorem ensures the existence
of a disc-like neighbourhood $D\subset M'$
of a point such that $\xi'|_D=\xi_{\st}$.
The disc $D$ can be isotoped to a collar extension of $D^3$
inside $M'$ by the smooth disc theorem.
Restricting the isotoped contact structure $\xi'$ to $M$
yields $\xi$.
 
Assuming the existence of an aperiodic $\xi$-defining
contact form $\alpha$ that is Euclidean near $\partial M$
Eliashberg--Hofer \cite{eh94} showed that $(M,\rmd\alpha)$
is odd-symplectomorphic to $(D^3,\rmd x\wedge\rmd y)$.
The global Darboux theorem from \cite{eh94} yields a contact embedding
of $(M,\xi)$ into $(\R^3,\xi_{\st})$
so that $\xi$ in particular is tight.
Consequently,
by Eliashberg's classification of tight contact structures on $D^3$
with standard characteristic foliation on $\partial D^3=S^2$,
$(M,\xi)$ is contactomorphic to $(D^3,\xi_{\st})$,
cf. \cite[Theorem 4.10.1 (b)]{gei08}.
Finally,
using Smale--Munkres' $\Gamma_3=0$,
$M'$ is diffeomorphic to $S^3$.


\subsection{Euclidean model for transverse knots\label{sec:aeuklmod}}

We consider $S^3$ as subset of $\C^2$
given by $\{|z_1|^2+|z_2|^2=1\}$.
Denote the standard contact structure
$TS^3\cap\rmi TS^3$ by $\xi_{\st}$,
which can be defined as the kernel of
$r_1^2\rmd\theta_1+r_2^2\rmd\theta_2$.
The $2$-torus in $S^3$ given by $\{|z_1|^2=\tfrac12\}$
divides the $3$-sphere into two solid tori
$\{|z_1|^2\geq\tfrac12\}$
and $\{|z_2|^2\geq\tfrac12\}$.
The standard contact structure $\xi_{\st}$ on the model solid torus
$S^1\times D^2$ is given by
$\ker\big(\rmd\tau+\tfrac12r^2\rmd\theta\big)$.
Both contact manifolds are contactomorphic
in such a way that the transverse knot $S^1\times\{0\}$
corresponds to $S^3\cap\big(\C\times\{z_2=0\}\big)$.
Indeed, the map $S^1\times D^2_{\sqrt{2}}\ra S^3$
given by
\[
(\tau,z)\longmapsto
\frac{\rme^{\rmi\tau}}{\sqrt{2}}
\Big(\sqrt{2-|z|^2},z\Big)
\]
defines a strict contactomorphism
\[
\Big(S^1\times D^2,\rmd\tau+\tfrac12r^2\rmd\theta\Big)
\lra
\Big(
\{|z_1|^2\geq\tfrac12\}\cap S^3,
r_1^2\rmd\theta_1+r_2^2\rmd\theta_2
\Big)
\,.
\]
The contact stereographic projection used in \cite[Proposition 2.1.8]{gei08}
defines a {\bf positive}, i.e.\ (co-)orientation preserving,
contactomorphism from the $3$-sphere with the north pole $(0,1)\in\C^2$ deleted
to $\R^3$ w.r.t.\ the standard contact structures.
A composition of both contactomorphisms
yields a positive contact embedding
\[
(S^1\times D^2,\xi_{\st})\lra(\R^3,\xi_{\st})
\,,
\]
whose image we denote by $V$.
The transverse knot $S^1\times\{0\}$
in the Euclidean model corresponds to 
$\partial D^2\times\{0\}$ in $\R^2\times\R$.

\begin{lem}
 \label{bvsatisatoc}
  $\partial V$ satisfies the requirements (a)-(c) of Definition \ref{defofboundcon}.
\end{lem}

\begin{proof}
The contact stereographic projection
in \cite[Proposition 2.1.8]{gei08}
is a composition of the classical
stereographic projection
$(z_1,z_2)\mapsto(1-y_2)^{-1}(x_1,y_1,x_2)$
and a correction map
\[
\big(r\rme^{\rmi\theta},t\big)
\longmapsto
\Big(r\rme^{\rmi(\theta-t)},\tfrac12\big(1+\tfrac13t^2+r^2\big)t\Big)
\]
w.r.t.\ the splitting $\C\times\R$,
cf.\ the first formula on \cite[p.~57]{gei08}.
The classical stereographic projection maps the solid torus
$\{|z_1|^2\geq\tfrac12\}\cap S^3$
onto the rotationally invariant solid torus $V'$
obtained by rotating the closed unit disc
in the $xz$-plane with centre $(\sqrt2,0)$
around the $z$-axis.
This can be seen using conformal properties of the stereographic projection.
The correction map is a composition
of the graph-map of rotations
$(z,t)\mapsto(\rme^{-\rmi t}z,t)$,
that leave the solid torus $V'$ invariant,
and dilations
$(z,t)\mapsto\big(z,p(|z|,t)\cdot t\big)$
along lines parallel to the $\R$-axis
letting $p(r,t)=\frac12\big(1+\tfrac13t^2+r^2\big)$.
Therefore,
parallel lines to the $\R$-axis
through the annulus
$A=\{\sqrt2-1\leq|z|\leq\sqrt2+1\}\times\{0\}$
in $\C\times\R$ intersect the $2$-torus $\partial V$
either transversally in precisely two points
or in a unique point of tangency along the circles
$\{\sqrt2-1=|z|\}$ and $\{\sqrt2+1=|z|\}$.
As the projection of $V$ onto $\C$ along $\R$
equals $A$ there are no other intersection points.
\end{proof}


\subsection{Exterior of a transverse knot\label{sec:exoftrknot}}

We consider a closed connected
contact $3$-manifold $(M',\xi')$
and a transverse knot $K\subset(M',\xi')$.
By the tubular neighbourhood theorem
there exist $r>0$
and a contact embedding
$(S^1\times B_r(0),\xi_{\st})\ra(M',\xi')$
that maps $S^1\times\{0\}$ onto $K$,
see \cite{gei08}.
Moreover, for any given $N\in\N$
the diffeomorphism 
\[
(\tau,z)\longmapsto
\left(\tau,\frac{\rme^{-N\rmi\tau}z}{\sqrt{1+\tfrac{N}{2}|z|^2}}\right)
\]
shrinks $S^1\times\R^2$ onto $S^1\times B_{2/N}(0)$
contactomorphically w.r.t.\ $\xi_{\st}$,
see \cite{eli91,ekp06}.
Composing both maps suitably,
we obtain a contact embedding
\[
f\co(S^1\times\R^2,\xi_{\st})\lra(M',\xi')
\]
that is positive and sends $S^1\times\{0\}$ to $K$.
The image of the restriction of $f$ to $S^1\times D^2$
is the tubular neighbourhood
\[
\nu K=f(S^1\times D^2)\subset M'\,,
\]
which is unique up to smooth isotopies.
The {\bf exterior} of the transverse knot $K\subset(M',\xi')$
(more precisely the exterior of the embedding $f$)
is defined to be
\[
M:=M'\setminus\Int(\nu K)\,.
\]
We equip $M$ with the contact structure $\xi:=\xi'|_M$
and provide $\partial M$ with the boundary orientation
so that $\partial(\nu K)=-\partial M$.

As in the \cite[Appendix]{bschz}
we obtain local contact inversions near transverse knots:

\begin{prop}
\label{loccontinv}
 $(M',\xi')$ admits a {\bf local contact inversion}
 along $\partial M$,
 i.e.\ there exists a positive contactomorphism
 defined near $\partial M\subset M'$ that preserves $\partial M$
 reversing the orientation.
\end{prop}

\begin{proof}
 The germ of the local contact involution is obtained
 from the involution
 $\iota(z_1,z_2)=(z_2,z_1)$ on $S^3\subset\C^2$,
 which positively preserves $\xi_{\st}$ and interchanges
 the solid tori
 $\{|z_1|^2\geq\tfrac12\}\cap S^3$ and
 $\{|z_2|^2\geq\tfrac12\}\cap S^3$.
 By the tubular neighbourhood theorem
 as formulated at the beginning of Section \ref{sec:exoftrknot}
 and the transverse model neighbourhood
 inside $(S^3,\xi_{\st})$ as considered in
 Section \ref{sec:aeuklmod}
 (with $z_1$ and $z_2$ interchanged)
 we obtain a positive contactomorphism
 \[
 \chi\co\big(\nu K,\xi'\big)\lra
 \Big(\{|z_2|^2\geq\tfrac12\}\cap S^3,\xi_{\st}\Big)
 \]
 from (a neighbourhood of) $\nu K$ onto (a
 neighbourhood of) $\{|z_2|^2\geq\tfrac12\}\cap S^3$
 w.r.t.\ $\xi'$ and $\xi_{\st}$ that mapps $K$
 onto $\{z_1=0\}\cap S^3$.
 The claim follows by pulling back $\iota$ along $\chi$.
\end{proof}

Composing $\chi$ with the contact stereographic projection
as in Section \ref{sec:aeuklmod} we obtain with Lemma \ref{bvsatisatoc}:

\begin{cor}
 The boundary $\partial M\subset (M,\xi)$ 
 of any knot exterior admits a
 contact embedding into $(\R^3,\xi_{\st})$.
\end{cor}


\subsection{The unknot in $S^3$\label{sec:theunknots3}}

Let $K$ be a knot in a closed connected oriented $3$-manifold $M'$.
By Martinet's theorem $M'$ admits infinitely many
positively co-oriented contact structures $\xi'$
such that $K$ is a transverse knot,
cf.\ \cite[Theorem 4.1.2]{gei08}.
In other words,
any knot appears as a transverse knot
for a certain contact structure.
On the other hand the following theorem
characterises the unknot in $S^3$ uniquely
in terms of symplectic dynamics on the knot exterior:

\begin{thm}
 \label{thm:diffeotos3}
 Let $K$ be a transverse knot
 in a closed connected co-oriented contact $3$-manifold $(M',\xi')$.
 We assume that
 the knot exterior $(M,\xi)$ of $K\subset (M',\xi')$
 admits an aperiodic $\xi$-defining contact form $\alpha$
 such that $\alpha$ is Euclidian near the boundary $\partial M$.
 Then
 $(M',K)$ is diffeomorphic to $(S^3,\{z_1=0\})$
 with orientations preserved
 so that $K$ (in particular) is the unknot.
\end{thm}

\begin{proof}
 The positive contactomorphism $\chi$
 obtained in the proof of Proposition \ref{loccontinv}
 embeds a collar neighbourhood
 $\big((-\varepsilon,0]\times\partial M,\xi\big)$
 into
 $\big(\{|z_1|^2\geq\tfrac12\}\cap S^3,\xi_{\st}\big)$
 contactomorphically
 (mapping $\partial M\equiv\{0\}\times\partial M$
 onto $\{|z_1|^2=\tfrac12\}\cap S^3$),
 which by the contact stereographic projection
 is contactomorphic to the solid torus $V\subset(\R^3,\xi_{\st})$,
 see Section \ref{sec:aeuklmod}.
 In view of Lemma \ref{bvsatisatoc},
 Proposition \ref{prop:aperieukl}
 and Section \ref{sec:tightcontsoltor} there exists an
 odd-symplectomorphism
 $F\co(M,\rmd\alpha)\ra(V,\rmd x\wedge\rmd y)$
 that fixes a longitudinal curve $\ell$ in $\partial M=T^2$.
 We choose a meridional curve by intersecting
 $\partial V$ with the positive half $xz$-plane $x>0$.
 Therefore,
 the restriction $F|_{\partial M\equiv\partial V}$
 is a $k$-fold Dehn twist along $\ell$ up to isotopy.
 Identifying $V$ with
 $\{|z_1|^2\geq\tfrac12\}\cap S^3$
 using contact stereographic projection,
 the map $F|_{\partial M\equiv\partial V}$
 extends to a diffeomorphism $\Delta$
 of the complement
 $\{|z_2|^2\geq\tfrac12\}\cap S^3$
 preserving the unknot $\{z_1=0\}\cap S^3$.
 Consequently,
 the diffeomorphisms $F$ and $\Delta\circ\chi$
 glue to an orientation preserving diffeomorphism $M'\ra S^3$.
\end{proof}


\begin{ack}
  This research is part of a project in the SFB/TRR 191
  {\it Symplectic Structures in Geometry, Algebra and Dynamics},
  funded by the DFG.
\end{ack}


\end{document}